\documentclass{amsart}
\usepackage{amsmath}
\usepackage{amsfonts}

\newtheorem{theorem}{Theorem}
\theoremstyle{plain}

\newtheorem{corollary}{Corollary}

\newtheorem{proposition}{Proposition}

\numberwithin{equation}{section}

\begin{document}
\title[]{Curvature Properties of Normal Complex Contact Metric Manifolds }
\author{Aysel TURGUT\ VANLI\ }
\address{Departments of Mathematics , Gazi University.}
\email{a.vanli@gazi.edu.tr}
\author{\.{I}nan \"{U}NAL\ }
\address{Departments of Computer Engineering, Tunceli Universty }
\email{inanunal@tunceli.edu.tr}
\subjclass[2010]{53C15, 53C25}
\keywords{normality, normal complex contact metric manifold, Ricci
curvature, curvature}

\begin{abstract}
In this paper, we study normal complex contact metric manifolds and we get
some general results on them. Moreover, we obtained the general expression
of the curvature tensor field for arbitrary vector fields. Furthermore, we
show that the necessary and sufficient conditions to be normal a complex
contact metric manifold. Also, we give some new equailities for Ricci
curvature of normal complex contact metric manifolds.
\end{abstract}

\maketitle

\renewcommand{\subjclassname}{{\rm 2010} Mathematics Subject
Classification}

\section{\protect\bigskip \textbf{Introduction }}

The study of complex contact manifolds began with Kobayashi \cite{KOb59} and
Boothby \cite{Boo61}, \cite{Boo62} in the late 1950's and early 1960's. In
1965, Wolf studied homogeneous complex contact manifolds \cite{Wolf65}.
Further research started again in the early 1980's by Ishihara and Konishi 
\cite{IK79}, \cite{IK80} and \cite{IK82}. They introduced a concept of
normality in \cite{IK80} but their normality condition seems too strong and
did not include natural examples like the complex Heisenberg group. In 1996,
Foreman investigate special metrics on complex contact manifolds by studying
critical condition of various Riemanian functionals on particular classes of
Riemanian metrics called the associated metrics \cite{FOR96}. He studied on
classification of three-dimensional complex homogeneous complex contact
manifolds, strict normal complex contact manifolds and the Boothby-Wang
Fibration on complex contact manifolds \cite{FOR99}, \cite{FOR2000} , \cite%
{FOR2000-2}. Then Korkmaz define a weaker version of normality in \cite%
{KB2000} and give the theorem which states the necessary and sufficient
conditions, in terms of the covariant derivatives of the structure tensors,
for a complex contact metric manifold to be normal. Korkmaz also defined the 
$\mathcal{GH}-$sectional curvature of normal complex contact metric
manifolds \cite{KB2000}. Blair and present author studied energy and
corrected energy of vertical distribution for normal complex contact metric
manifolds in \cite{VAN2004}, \cite{VAN2006}. Fetcu studied an adapted
connection on a strict complex contact manifolds and harmonic maps between
complex Sasakian manifolds in \cite{FER2006}, \cite{FERCU2006}.

Blair and Molina show that normal complex contact metric manifolds that are
Bochner flat must have constant holomorphic sectional curvature 4 and be K%
\"{a}hler \cite{BL2011}, and they showed that it is not possible for normal
complex contact metric manifolds to be conformally flat. In 2012, Blair and
Mihai prove that a complex $(\kappa ,%
\mu
)-$space with $\kappa <1$ is a locally homogeneous complex contact metric
manifold and they studied on locally symmetric condition of normal complex
contact metric manifolds \cite{BL2012}, \cite{BL2012-2}. For a general
discussion in complex contact geometry we refer to the reader to \cite%
{BL2010}.

In Section 2, we give the necessary definitions and some basic facts about
complex contact metric manifolds and also we give some curvature properties
of normal complex contact metric manifolds. In Section 3, we get some
general results on normal complex contact metric manifolds. Then in Section
4, we obtained the general expression of the curvature tensor field for
arbitrary vector fields. Furthermore, we give a new theorem for normality.
The theorem states the necessary and sufficient conditions, in terms of the
covariant derivatives of the structure tensors $G$ and $H$, for a complex
contact metric manifold to be normal. In Section 5, we discuss some results
on Ricci curvature of a normal complex contact metric manifold.

\section{\textbf{Preliminaries}}

\qquad A \textit{complex contact manifold} is a complex manifold of odd
complex dimension $2n+1$ together with an open covering $\{\mathcal{U\}}$ by
coordinate neighborhoods such that

\begin{enumerate}
\item On each $\mathcal{U}$, there is a holomorphic $1$-form $\theta $ with $%
\theta \wedge (d\theta )^{n}\neq 0$.

\item On $\mathcal{U}\cap \mathcal{U}^{\prime }\neq \emptyset $ there is a
non-vanishing holomorphic function $f$ such that $\theta ^{\prime }=f\theta $%
.
\end{enumerate}

The complex contact structure determines a non-integrable distribution $%
\mathcal{H}$ by the equation $\theta =0$. A complex contact structure is
given by a global 1-form if and only if its first Chern class vanishes \cite%
{Boo61}.

Let $M$ be a Hermitian manifold with almost complex structure $J$, Hermitian
metric $g$ and open covering by coordinate neighborhoods $\{\mathcal{U\}}$; $%
M$ is called a \textit{complex almost contact metric manifold} if it
satisfies the following two conditions:

\begin{enumerate}
\item In each $\mathcal{U}$ there exist $1$-forms $u$ and $v=u\circ J$, with
dual vector fields $U$ and $V=-JU$ and $(1,1)$ tensor fields $G$ and $H=GJ$
such that 
\begin{equation*}
H^{2}=G^{2}=-I+u\otimes U+v\otimes V
\end{equation*}%
\begin{equation*}
GJ=-JG,\quad GU=0,\quad g(X,GY)=-g(GX,Y).
\end{equation*}

\item On $\mathcal{U}\cap \mathcal{U}^{\prime }\neq \emptyset $ we have 
\begin{eqnarray*}
u^{\prime } &=&au-bv,\quad v^{\prime }=bu+av,\; \\
G^{\prime } &=&aG-bH,\quad H^{\prime }=bG+aH
\end{eqnarray*}%
where $a$ and $b$ are functions on $\mathcal{U}\cap \mathcal{U}^{\prime }$
with $a^{2}+b^{2}=1$.
\end{enumerate}

Since $u$ and $v$ are dual to the vector fields $U$ and $V$, we easily see
from the second condition that on $\mathcal{U}\cap \mathcal{U}^{\prime }$, $%
U^{\prime }=aU-bV$ and $V^{\prime }=bU+aV$. Also since $a^{2}+b^{2}=1$, $%
U^{\prime }\wedge V^{\prime }=U\wedge V$. Thus $U$ and $V$ determine a
global vertical distribution $\mathcal{V}$ by $\xi =U\wedge V$ which is
typically assumed to be integrable.

A result of this definition, on a complex almost contact metric manifold $M,$
the following identities hold;%
\begin{eqnarray*}
HG &=&-GH=J+u\otimes V-v\otimes U \\
JH &=&-HJ=G \\
g(HX,Y) &=&-g(X,HY) \\
GU &=&HU=HV=0 \\
uG &=&vG=uH=vH=0 \\
JV &=&U,~~g(U,V)=0
\end{eqnarray*}

A complex contact manifold admits a complex almost contact metric structure
for which the local contact form $\theta $ is $u-iv$ to within a
non-vanishing complex-valued function multiple and the local fields $G$ and $%
\ H$ are related to $du$ and $dv$ by 
\begin{eqnarray*}
du(X,Y) &=&g(X,GY)+(\sigma \wedge v)(X,Y),~~~ \\
dv(X,Y) &=&g(X,HY)-(\sigma \wedge u)(X,Y)
\end{eqnarray*}%
where $\sigma (X)=g(\nabla _{X}U,V)$, $\nabla $ being the Levi-Civita
connection of $g$ (\cite{IK82}, \cite{FOR96}).We refer to a complex contact
manifold with a complex almost contact metric structure satisfying these
conditions as a \textit{complex contact metric manifold}.

On a complex contact metric manifold \cite{KB2000} 
\begin{equation}
\nabla _{U}G=\sigma (U)H\text{ \ \ and \ \ }\nabla _{V}H=-\sigma (V)G.
\end{equation}%
Ishihara and Konishi \cite{IK79}, \cite{IK80} introduced a notion of
normality for complex contact structures. Their notion is the vanishing of
the two tensor fields $S$ and $T$ given by 
\begin{eqnarray}
S(X,Y) &=&[G,G](X,Y)+2g(X,GY)U-2g(X,HY)V \\
&&+2(v(Y)HX-v(X)HY)+\sigma (GY)HX  \notag \\
&&-\sigma (GX)HY+\sigma (X)GHY-\sigma (Y)GHX,  \notag
\end{eqnarray}%
\begin{eqnarray}
T(X,Y) &=&[H,H](X,Y)-2g(X,GY)U+2g(X,HY)V \\
&&+2(u(Y)GX-u(X)GY)+\sigma (HX)GY  \notag \\
&&-\sigma (HY)GX+\sigma (X)GHY-\sigma (Y)GHX.  \notag
\end{eqnarray}%
We recall that%
\begin{equation*}
\lbrack G,G](X,Y)=(\nabla _{GX}G)Y-(\nabla _{GY}G)X-G(\nabla
_{X}G)Y+G(\nabla _{Y}G)X
\end{equation*}%
is the Nijenhuis torsion of $G$.

Korkmaz \cite{KB2000} give somewhat weaker definition; a complex contact
metric manifold is normal if

$\qquad S(X,Y)=T(X,Y)=0$ \ for all $X,Y$ in $\mathcal{H},$ and $\ $

$\qquad S\left( X,U\right) =T\left( X,V\right) =0$ for all $X.$

It was also proved in \cite{KB2000} that;

\begin{proposition}
Let M be a complex contact metric manifold. Then $M$ is normal if and only
if\bigskip 
\begin{eqnarray}
g((\nabla _{X}G)Y,Z) &=&\sigma (X)g(HY,Z)+v(X)d\sigma (GZ,GY) \\
&&-2v(X)g(HGY,Z)-u(Y)g(X,Z)  \notag \\
&&-v(Y)g(JX,Z)+u(Z)g(X,Y)  \notag \\
&&+v(Z)g(JX,Y),  \notag
\end{eqnarray}%
\begin{eqnarray}
g((\nabla _{X}H)Y,Z) &=&-\sigma (X)g(GY,Z)-u(X)d\sigma (HZ,HY) \\
&&-2u(X)g(HGY,Z)+u(Y)g(JX,Z)  \notag \\
&&-v(Y)g(X,Z)-u(Z)g(JX,Y)  \notag \\
&&+v(Z)g(X,Y).  \notag
\end{eqnarray}
\end{proposition}

As a result of this proposition, on a normal complex contact metric
manifold, the covariant derivative of $J$ satisfies%
\begin{eqnarray}
g((\nabla _{X}J)Y,Z) &=&u(X)(d\sigma (Z,GY)-2g(HY,Z)) \\
&&+v(X)(d\sigma (Z,HY)+2g(GY,Z)).  \notag
\end{eqnarray}%
Also in \cite{KB2000} on a normal complex contact metric manifold we
have\bigskip 
\begin{eqnarray}
\nabla _{X}U &=&-GX+\sigma (X)V,~~ \\
\ \ \ \nabla _{X}V &=&-HX-\sigma (X)U,~~  \notag
\end{eqnarray}%
\begin{eqnarray}
\nabla _{U}U &=&\sigma (U)V,~~~\nabla _{U}V=-\sigma (U)U \\
\nabla _{V}U &=&\sigma (V)V,~~~~\nabla _{V}V=-\sigma (V)U,~~  \notag
\end{eqnarray}%
\begin{eqnarray}
d\sigma (GX,GY) &=&d\sigma (HX,HY) \\
&=&d\sigma (Y,X)-2u\wedge v(Y,X)d\sigma (U,V),~  \notag
\end{eqnarray}%
\begin{eqnarray}
d\sigma (U,X) &=&v(X)d\sigma (U,V),~~~~~~ \\
\ d\sigma (V,X) &=&-u(X)d\sigma (U,V).  \notag
\end{eqnarray}

Our convention for the curvature is 
\begin{equation*}
R(X,Y)Z=\bigtriangledown _{X}\bigtriangledown _{Y}Z-\bigtriangledown
_{Y}\bigtriangledown _{X}Z-\bigtriangledown _{\lbrack X,Y]}Z
\end{equation*}%
\begin{equation*}
R(X,Y,Z,W)=g(R(X,Y)Z,W).
\end{equation*}

For a general discussion of complex contact manifolds we refer to (\cite%
{BL2010},Chapter 12). We will also need the basic curvature properties of
normal complex contact metric manifolds which we list here \cite{KB2000}.
First of all for $U$ and $V=-JU$ \ vertical vector fields we have

\begin{equation}
R(U,V,V,U)=R(V,U,U,V)=-2d\sigma (U,V)~~
\end{equation}%
For $X$ and $Y$ horizontal vector fields we have the following 
\begin{equation}
R(X,U)U=X,~~~~~~R(X,V)V=X
\end{equation}%
\begin{equation}
R(X,Y)U=2(g(X,JY)+d\sigma (X,Y))V
\end{equation}%
\begin{equation}
R(X,Y)V=-2(g(X,JY)+d\sigma (X,Y))U
\end{equation}%
\begin{equation}
R(X,U)V=\sigma (U)GX+(\bigtriangledown _{U}H)X-JX
\end{equation}%
\begin{equation}
R(X,V)U=-\sigma (V)HX+(\nabla _{V}G)X+JX
\end{equation}

\begin{equation}
R(X,U)Y=-g(X,Y)U-g(JX,Y)V+d\sigma (Y,X)V,
\end{equation}%
\begin{equation}
R(X,V)Y=-g(X,Y)V+g(JX,Y)U-d\sigma (Y,X)U~\ \ 
\end{equation}%
\begin{equation}
R(U,V)X=JX\text{ .}
\end{equation}%
In \cite{KB2000}, we have \ \ \ \ \ 
\begin{eqnarray}
g(R(GX,GY)GZ,GW) &=&g(R(X,Y)Z,W) \\
&&-2g(JZ,W)d\sigma (X,Y)  \notag \\
&&+2g(HX,Y)d\sigma (GZ,W)  \notag \\
&&+2g(JX,Y)d\sigma (Z,W)  \notag \\
&&-2g(HZ,W)d\sigma (GX,Y),  \notag
\end{eqnarray}

\begin{eqnarray}
g(R(HX,HY)HZ,HW) &=&g(R(X,Y)Z,W) \\
&&-2g(JZ,W)d\sigma (X,Y)  \notag \\
&&-2g(GX,Y)d\sigma (HZ,W)  \notag \\
&&+2g(JX,Y)d\sigma (Z,W)  \notag \\
&&+2g(GZ,W)d\sigma (HX,Y).  \notag
\end{eqnarray}

On the other hand , in \cite{FOR96} we get 
\begin{equation}
d\sigma (X,Y)=2g(JX,Y)+g((\bigtriangledown _{U}J)GX,Y).
\end{equation}

\section{\textbf{Some General Results On Normal Complex Contact Metric
Manifolds}}

\begin{theorem}
Let $M$ be a normal complex contact metric manifold\ and $X,Y$ arbitrary
vector fields on $M$. Then we have 
\begin{equation*}
\left( \bigtriangledown _{X}u\right) Y=g(X,GY)+\sigma \left( X\right)
v\left( Y\right)
\end{equation*}%
\begin{equation*}
\left( \bigtriangledown _{X}v\right) Y=g(X,GY)-\sigma \left( X\right)
u\left( Y\right)
\end{equation*}
\end{theorem}

\begin{proof}
For $X$ vector field, $1-$forms $u$ and $v$ we get 
\begin{equation*}
\left( \bigtriangledown _{X}u\right) Y=\bigtriangledown _{X}g(U,Y)-g\left(
U,\bigtriangledown _{X}Y\right) =g\left( \bigtriangledown _{X}U,Y\right)
\end{equation*}%
and 
\begin{equation*}
\left( \bigtriangledown _{X}v\right) Y=\bigtriangledown _{X}g(V,Y)-g\left(
V,\bigtriangledown _{X}Y\right) =g\left( \bigtriangledown _{X}V,Y\right)
\end{equation*}%
by using (2.7)%
\begin{equation*}
\left( \bigtriangledown _{X}u\right) Y=g(-GX+\sigma (X)V,Y)\text{ and }%
\left( \bigtriangledown _{X}v\right) Y=g\left( -HX-\sigma (X)U,Y\right)
\end{equation*}%
so the proof is completed.
\end{proof}

\begin{corollary}
Let M be a normal complex contact metric manifold. Then we have 
\begin{eqnarray}
~g\left( \left( \bigtriangledown _{U}G\right) X_{0},V\right) &=&0~~,~g\left(
\left( \bigtriangledown _{U}H\right) X_{0},V\right) =0 \\
~g\left( \left( \bigtriangledown _{U}G\right) X_{0},U\right) &=&0~\
~,~g\left( \left( \bigtriangledown _{U}H\right) X_{0},U\right) =0  \notag \\
g\left( \left( \bigtriangledown _{V}G\right) X_{0},U\right) &=&0~,~g\left(
\left( \bigtriangledown _{V}H\right) X_{0},U\right) =~0~  \notag \\
~g\left( \left( \bigtriangledown _{V}G\right) X_{0},V\right) &=&0~~,~g\left(
\left( \bigtriangledown _{V}H\right) X_{0},V\right) =~0~~  \notag
\end{eqnarray}%
\begin{eqnarray}
~g\left( \left( \bigtriangledown _{U}J\right) X_{0},V\right) &=&0~~,~g\left(
\left( \bigtriangledown _{U}J\right) X_{0},U\right) =0~\  \\
g\left( \left( \bigtriangledown _{V}J\right) X_{0},U\right) &=&0~,~~g\left(
\left( \bigtriangledown _{V}J\right) X_{0},V\right) =0  \notag
\end{eqnarray}%
for all $X_{0}\in \mathcal{H}$ .
\end{corollary}

\begin{proof}
For $X_{0}\in \mathcal{H}$ and from (2.4), (2.5), (2.6) we have (3.1) and
(3.2).
\end{proof}

\begin{theorem}
Let M be a normal complex contact metric manifold. Then, we have%
\begin{equation}
\left( \nabla _{U}G\right) X_{0}=\left( \nabla _{U}G\right) X\,~~,\left(
\nabla _{V}G\right) X_{0}=\left( \nabla _{V}G\right) X,\,
\end{equation}%
\begin{equation}
~\left( \nabla _{U}H\right) X_{0}=\left( \nabla _{U}H\right) X~,\left(
\nabla _{V}H\right) X_{0}=\left( \nabla _{V}H\right) X,
\end{equation}%
\begin{equation}
\left( \nabla _{U}J\right) X_{0}=\left( \nabla _{U}J\right) X\,~~,\left(
\nabla _{V}J\right) X_{0}=\left( \nabla _{V}J\right) X,
\end{equation}%
\begin{equation}
g\left( \left( \bigtriangledown _{U}G\right) X,Y\right) =g\left( \left(
\bigtriangledown _{U}G\right) X_{0},Y_{0}\right) =\sigma (U)g(HX_{0},Y_{0}),
\end{equation}%
\begin{eqnarray}
g\left( \left( \bigtriangledown _{V}G\right) X,Y\right) &=&g\left( \left(
\bigtriangledown _{V}G\right) X_{0},Y_{0}\right) \\
&=&\sigma (V)g(HX_{0},Y_{0})+d\sigma (Y_{0},X_{0})-2g(JX_{0},Y_{0}),  \notag
\end{eqnarray}%
\begin{equation}
g\left( \left( \bigtriangledown _{V}H\right) X,Y\right) =g\left( \left(
\bigtriangledown _{V}H\right) X_{0},Y_{0}\right) =-\sigma (V)g(GX_{0},Y_{0})
\end{equation}%
\begin{eqnarray}
g\left( \left( \bigtriangledown _{U}H\right) X,Y\right) &=&g\left( \left(
\bigtriangledown _{U}H\right) X_{0},Y_{0}\right) \\
&=&-\sigma (U)g(GX_{0},Y_{0})  \notag \\
&&-d\sigma (Y_{0},X_{0})+2g(JX_{0},Y_{0})  \notag
\end{eqnarray}%
\begin{eqnarray}
g(\left( \nabla _{U}J\right) GX,Y) &=&g(\left( \nabla _{U}J\right)
GX_{0},Y_{0}) \\
&=&-d\sigma (Y_{0},X_{0})-2g(JX_{0},Y_{0})  \notag
\end{eqnarray}%
\begin{eqnarray}
g(\left( \nabla _{V}J\right) GX,Y) &=&g(\left( \nabla _{V}J\right)
GX_{0},Y_{0}) \\
&=&d\sigma (Y_{0},GX_{0})-2g(HX_{0},Y_{0})  \notag
\end{eqnarray}%
where $X=X_{0}+u(X)U+v(X)V$ , $Y=Y_{0}+u(Y)U+v(Y)V$ and $X_{0}$ ,$Y_{0}$ $%
\in \mathcal{H}.$

\begin{proof}
For $X=X_{0}+u(X)U+v(X)V$ we have 
\begin{eqnarray*}
\left( \nabla _{U}G\right) X &=&\nabla _{U}GX_{0}-G\left( \nabla
_{U}X_{0}+\nabla _{U}u(X\right) U+\nabla _{U}v(X)V) \\
&=&\left( \nabla _{U}G\right) X_{0}-G\left( U\left[ u(X)\right] U+u(X)\nabla
_{U}U\right. \\
&&\left. +U\left[ v(X)\right] V+v(X)\nabla _{U}V\right)
\end{eqnarray*}%
and%
\begin{eqnarray*}
\left( \nabla _{U}H\right) X &=&\nabla _{U}HX-H\left( \nabla _{U}X\right) \\
&=&\nabla _{U}HX_{0}-H\left( \nabla _{U}X_{0}+\nabla _{U}u(X\right) U+\nabla
_{U}v(X)V) \\
&=&\left( \nabla _{U}H\right) X_{0}-H\left( U\left[ u(X)\right] U+u(X)\nabla
_{U}U\right. \\
&&\left. +U\left[ v(X)\right] V+v(X)\nabla _{U}V\right)
\end{eqnarray*}%
since $GU=GV=HU=HV=0$ and from (2.8)\ we have $\left( \nabla _{U}G\right)
X_{0}=\left( \nabla _{U}G\right) X~~~$and $\ \left( \nabla _{U}H\right)
X_{0}=\left( \nabla _{U}H\right) X.$ By similar way we can proof the others,
so we get (3.3) and (3.4) . On the other hand 
\begin{eqnarray*}
\left( \nabla _{U}J\right) X &=&\nabla _{U}\left( JX_{0}+u(X)U+v(X)V\right)
-J\left( \nabla _{U}X_{0}\right. \\
&&+\nabla _{U}u(X)U+\nabla _{U}v(X)V) \\
&=&\left( \nabla _{U}J\right) X_{0}+U\left[ u(X)\right] U+u(X)\nabla _{U}U \\
&&\left. +U\left[ v(X)\right] V+v(X)\nabla _{U}V\right) -J\left( U\left[ u(X)%
\right] U\right. \\
&&+u(X)\nabla _{U}U\left. +U\left[ v(X)\right] V+v(X)\nabla _{U}V\right)
\end{eqnarray*}%
from (2.8)\ and since $V=-JU,U=JV$ we obtain $\left( \nabla _{U}J\right)
X=\left( \nabla _{U}J\right) X_{0}.$ Similarly we can proof $\left( \nabla
_{V}J\right) X=\left( \nabla _{V}J\right) X_{0}$, so we get (3.5). In
addition for $Y=Y_{0}+u(Y)U+v(Y)V$ from (2.4), (3.1) and (3.3) we have 
\begin{eqnarray*}
g((\nabla _{U}G)X,Y) &=&g((\nabla _{U}G)X_{0},Y_{0})=\sigma
(U)g(HX_{0},Y_{0}) \\
&&+v(U)d\sigma (GY_{0},GX_{0})-2v(U)g(HGX_{0},Y_{0}) \\
&&-u(X_{0})g(U,Y_{0})-v(X_{0})g(JU,Y_{0})+u(Y_{0})g(U,X_{0}) \\
&&+v(Y_{0})g(JU,X_{0}) \\
&=&\sigma (U)g(HX_{0},Y_{0})
\end{eqnarray*}%
and 
\begin{eqnarray*}
g((\nabla _{V}G)X,Y) &=&g((\nabla _{V}G)X_{0},Y_{0})=\sigma
(V)g(HX_{0},Y_{0}) \\
&&+v(V)d\sigma (GY_{0},GX_{0})-2v(V)g(HGX_{0},Y_{0}) \\
&&-u(X_{0})g(V,Y_{0})-v(X_{0})g(JV,Y_{0})+u(Y_{0})g(V,X_{0}) \\
&&+v(Y_{0})g(JV,X_{0}) \\
&=&\sigma (V)g(HX_{0},Y_{0})+d\sigma (Y_{0},X_{0})-g(JX_{0},Y_{0})
\end{eqnarray*}%
so we get (3.6) and (3.7). By the same way from (2.5), (3.1) and (3.4) we
have 
\begin{eqnarray*}
g((\nabla _{V}H)X,Y) &=&g((\nabla _{V}H)X_{0},Y_{0})=-\sigma
(V)g(GX_{0},Y_{0})-u(V)d\sigma (HY_{0},HX_{0}) \\
&&-2u(V)g(HGX_{0},Y_{0})+u(X_{0})g(JV,Y_{0}) \\
&&-v(X_{0})g(V,Y_{0})-u(Y_{0})g(JV,X_{0}) \\
&&+v(Y_{0})g(V,X_{0}) \\
&=&-\sigma (V)g(GX_{0},Y_{0})
\end{eqnarray*}%
and 
\begin{eqnarray*}
g((\nabla _{U}H)X,Y) &=&g((\nabla _{U}H)X_{0},Y_{0})=-\sigma
(U)g(GX_{0},Y_{0})-u(U)d\sigma (HY_{0},HX_{0}) \\
&&-2u(U)g(HGX_{0},Y_{0})-u(X_{0})g(JU,Y_{0}) \\
&&-v(X_{0})g(U,Y_{0})-u(Y_{0})g(JU,X_{0}) \\
&&+v(Y_{0})g(U,X_{0}) \\
&=&-\sigma (U)g(GX_{0},Y_{0})-d\sigma (Y_{0},X_{0})+2g(JX_{0},Y_{0})
\end{eqnarray*}%
so we get (3.8) and (3.9) . Finally \ from (2.6) , (3.2) and (3.5) we get
(3.10) and (3.11).
\end{proof}
\end{theorem}

\section{\protect\bigskip \textbf{Some Results on Curvature of Normal
Complex Contact Metric Manifolds}}

\begin{theorem}
Let M be a normal complex contact metric manifold. For $X,Y,Z,W$ horizontal
vector fields we have 
\begin{equation}
g(R(GX,GY)GZ,GW)=g(R(HX,HY)HZ,HW)=g(R(X,Y)Z,W).
\end{equation}

\begin{proof}
By using (2.22) in (2.20)\ we have%
\begin{eqnarray*}
&&-2g(JZ,W)d\sigma (X,Y)+2g(HX,Y)d\sigma (GZ,W) \\
&&+2g(JX,Y)d\sigma (Z,W)-2g(HZ,W)d\sigma (GX,Y) \\
&=&-2g(JZ,W)\left( 2g(JX,Y)+g((\bigtriangledown _{U}J)GX,Y)\right) \\
&&+2g(HX,Y)\left( 2g(JGZ,W)+g((\bigtriangledown _{U}J)G^{2}Z,W)\right) \\
&&+2g(JX,Y)\left( 2g(JZ,W)+g((\bigtriangledown _{U}J)GZ,W)\right) \\
&&-2g(HZ,W)(2g(JgX,Y)+g((\bigtriangledown _{U}J)G^{2}X,Y))
\end{eqnarray*}%
Since $JG=-H$ and for $X$ horizontal vector field $G^{2}X=-X$ \ we have%
\begin{eqnarray*}
&&-2g(JZ,W)d\sigma (X,Y)+2g(HX,Y)d\sigma (GZ,W) \\
&&+2g(JX,Y)d\sigma (Z,W)-2g(HZ,W)d\sigma (GX,Y)) \\
&=&-2g(JZ,W)g((\bigtriangledown _{U}J)GX,Y)-2g(HX,Y)g((\bigtriangledown
_{U}J)GZ,W) \\
&&+2g(JX,Y)g((\bigtriangledown _{U}J)GZ,W)+2g(HZ,W)g((\bigtriangledown
_{U}J)GX,Y))
\end{eqnarray*}%
From (2.6) and \ by simply computation we get%
\begin{eqnarray*}
&&-2g(JZ,W)d\sigma (X,Y)+2g(HX,Y)d\sigma (GZ,W) \\
&&+2g(JX,Y)d\sigma (Z,W)-2g(HZ,W)d\sigma (GX,Y)=0
\end{eqnarray*}%
Considering (2.20 ) from last equation we obtained that 
\begin{equation*}
g(R(GX,GY)GZ,GW)=g(R(X,Y)Z,W)
\end{equation*}%
By same way we can easily show that $g(R(HX,HY)HZ,HW)=g(R(X,Y)Z,W).$ So, the
proof is completed.
\end{proof}
\end{theorem}

\begin{theorem}
Let M be a normal complex contact metric manifold and $X,Y$ be arbitrary
vector fields on M. Then the curvature of a normal complex contact metric
manifold satisfies following equations. 
\begin{equation}
R(X,U)U=X_{0}-2d\sigma (U,V)v(X)V
\end{equation}%
\begin{equation}
R(X,V)V=X_{0}-2d\sigma (U,V)u(X)U
\end{equation}%
\begin{equation}
R(X,U)V=\sigma (U)GX_{0}+\left( \nabla _{U}H\right) X_{0}-JX_{0}+2d\sigma
(U,V)v(X)U
\end{equation}%
\begin{equation}
R(X,V)U=-\sigma (V)HX_{0}+\left( \nabla _{V}G\right) X_{0}+JX_{0}+2d\sigma
(U,V)u(X)V
\end{equation}%
\begin{equation}
R(U,V)X=JX_{0}+(u(x)V-v(X)U)2d\sigma (U,V)
\end{equation}%
\begin{eqnarray}
R(X,Y)U &=&-u(X)Y_{0}+v(X)\left( \sigma (V)HY_{0}+\left( \nabla _{V}G\right)
Y_{0}+JY_{0}\right) \\
&&+u(Y)X_{0}+v(Y)\left( -\sigma \left( V\right) HX_{0}+\left( \nabla
_{V}G\right) X_{0}+JX_{0}\right)  \notag \\
&&+\left[ 2\left( g(X_{0},JY_{0}\right) +d\sigma \left( X_{0},Y_{0}\right)
)\right.  \notag \\
&&\left. +2d\sigma \left( U,V\right) u\wedge v(X,Y)\right] V  \notag
\end{eqnarray}%
\begin{eqnarray}
R(X,Y)V &=&-u(X)\left( \sigma (U)GY_{0}+\left( \nabla _{U}H\right)
Y_{0}-JY_{0}\right) -v(X)Y_{0} \\
&&u(Y)\left( -\sigma \left( U\right) GX_{0}+\left( \nabla _{U}H\right)
X_{0}-JX_{0}\right) +v(Y)X_{0}+  \notag \\
&&\left[ -2\left( g(X_{0},JY_{0}\right) +d\sigma \left( X_{0},Y_{0}\right)
)\right.  \notag \\
&&\left. -2d\sigma \left( U,V\right) u\wedge v(X,Y)\right] U  \notag
\end{eqnarray}%
\begin{eqnarray}
R(X,U)Y &=&u\left( Y\right) X_{0}-v(X)JY_{0}+v(Y)\left( \sigma
(U)GX_{0}+\left( \nabla _{U}H\right) X_{0}\right. \\
&&\left. -JX_{0}\right) +\left[ -g(X_{0},Y_{0})-2d\sigma \left( U,V\right)
v(X)v(Y)\right] U  \notag \\
&&+\left[ d\sigma \left( Y_{0},X_{0}\right) -g(JX_{0},Y_{0})\right.  \notag
\\
&&-\left. 2d\sigma \left( U,V\right) v(X)u(Y))\right] V  \notag
\end{eqnarray}%
\begin{eqnarray}
R(X,V)Y &=&u(X)JY_{0}+v(Y)X_{0}+u(Y)\left( -\sigma (U)HX_{0}+\left( \nabla
_{V}G\right) X_{0}\right. \\
&&\left. +JX_{0}\right) +\left[ -g(X_{0},Y_{0})+u(X)u(Y)2d\sigma \left(
U,V\right) \right] V  \notag \\
&&\left[ -d\sigma \left( Y_{0},X_{0}\right) +g(JX_{0},Y_{0})\right.  \notag
\\
&&-\left. 2d\sigma \left( U,V\right) u(X)v(Y)\right] U  \notag
\end{eqnarray}%
where $X=X_{0}+u(X)U+v(X)V$ and $Y=Y_{0}+u(Y)U+v(Y)V.$

\begin{proof}
For $X=X_{0}+u(X)U+v(X)V$ \ we have%
\begin{equation*}
R(X,U)U=R(X_{0},U)U+u(X)R(U,U)U+v(X)R(V,U)U
\end{equation*}%
and%
\begin{equation*}
R(X,V)V=R(X_{0},U)V+u(X)R(U,V)V+v(X)R(V,V)V
\end{equation*}

From (2.12) and (2.13) \ we get (4.2) and (4.3).%
\begin{equation*}
R(X,U)V=R\left( X_{0},U\right) V+u(X)R(U,U)V+v(X)R(V,U)V
\end{equation*}%
and%
\begin{equation*}
R(X,V)U=R(X_{0},V)U+u(X)R(U,V)U+v(X)R(U,V)V
\end{equation*}%
and%
\begin{equation*}
R(U,V)X=R(U,V)X_{0}+u(X)R(U,V)U+v(X)R(U,V)V
\end{equation*}%
from (2.12), (2.15) , (2.16) and (2.19) we get (4.4) ,(4.5) and (4.6).

For $X=X_{0}+u(X)U+v(X)V$ and $Y=Y_{0}+u(Y)U+v(Y)V$ 
\begin{eqnarray*}
R(X,Y)U &=&R(X_{0},Y)U+u(X)R(U,Y)U+v(X)R\left( V,Y\right) U \\
&=&R(X_{0},Y_{0})U+u(Y)R(X_{0},U)U+v(Y)R(X_{0},V)U \\
&&+u(X)R(U,Y)U+v(X)R\left( V,Y\right) U
\end{eqnarray*}%
and%
\begin{eqnarray*}
R(X,Y)V &=&R(X_{0},Y)V+u(X)R(U,Y)V+v(X)R\left( V,Y\right) V \\
&=&R(X_{0},Y_{0})V+u(Y)R(X_{0},U)V+v(Y)R(X_{0},V)V \\
&&+u(X)R(U,Y)V+v(X)R\left( V,Y\right) V
\end{eqnarray*}%
from (2.13) , (2.14) ,(4.2)(4.3) ,(4.4) and (4.5) we get (4.7)\ and (4.8).
In addition 
\begin{eqnarray*}
R\left( X,U\right) Y &=&R\left( X_{0},U\right) Y+u(X)R\left( U,U\right)
Y+v(X)R\left( V,U\right) Y \\
&=&R\left( X_{0},U\right) Y_{0}+u(Y)R(X_{0},U)U+v(Y)R\left( X_{0},U\right) V
\\
&&+v(X)R\left( V,U\right) Y
\end{eqnarray*}%
and 
\begin{eqnarray*}
R\left( X,V\right) Y &=&R\left( X_{0},V\right) Y+u(X)R\left( U,V\right)
Y+v(X)R\left( V,V\right) Y \\
&=&R\left( X_{0},V\right) Y_{0}+u(Y)R(X_{0},V)U+v(Y)R\left( X_{0},V\right) V
\\
&&+u(X)R\left( U,V\right) Y
\end{eqnarray*}%
from (2.17)\ , (2.18) and (4.6) we have (4.9) and (4.10).
\end{proof}
\end{theorem}

\begin{corollary}
Let $X,Y,Z$ be horizontal vector fields on $M.$ Then the curvature of a
normal complex contact metric manifold satisfies equations (2.12), (2.13),
(2.14), (2.15), (2.16), (2.17), (2.18) and (2.19).
\end{corollary}

\begin{theorem}
Let M be a normal complex contact metric manifold $.$ Then for arbitrary
vector fields $X$ and $Y$ on M we have 
\begin{equation}
d\sigma \left( X,Y\right) =2g\left( JX_{0},Y_{0}\right) +g\left( \left(
\nabla _{U}J\right) GX_{0},Y_{0}\right) +d\sigma (U,V)u\wedge v(X,Y).
\end{equation}
\end{theorem}

\begin{proof}
For vector fields $X$ and $Y$ we have 
\begin{eqnarray*}
2d\sigma \left( X,Y\right) &=&X\sigma \left( Y\right) -Y\sigma \left(
X\right) -\sigma \left( \left[ X,Y\right] \right) \\
&=&Xg\left( \nabla _{Y}U,V\right) -Yg\left( \nabla _{X}U,V\right) -g\left(
\nabla _{\left[ X,Y\right] }U,V\right) \\
&=&g\left( \nabla _{X}\nabla _{Y}U,V\right) +g\left( \nabla _{Y}U,\nabla
_{X}V\right) -g\left( \nabla _{Y}\nabla _{X}U,V\right) \\
&&-g\left( \nabla _{X}U,\nabla _{Y}V\right) -g\left( \nabla _{\left[ X,Y%
\right] }U,V\right) \\
&=&g(R\left( X,Y\right) U,V)+g\left( \nabla _{Y}U,\nabla _{X}V\right)
-g\left( \nabla _{X}U,\nabla _{Y}V\right)
\end{eqnarray*}%
from (2.7) and since $HG=-GH=J+u\otimes V-v\otimes U$ we have 
\begin{equation*}
2d\sigma \left( X,Y\right) =g(R\left( X,Y\right) U,V)+2g(JX,Y)+2u\wedge
v(X,Y)
\end{equation*}%
In addition from (4.7) and for $X=X_{0}+u(X)U+v(X)V$ , $Y=Y_{0}+u(Y)U+v(Y)V$
we get 
\begin{equation*}
g(R\left( X,Y\right) U,V)=2\left( g(X_{0},JY_{0})+d\sigma
(X_{0},Y_{0})\right) +2d\sigma (U,V)u\wedge v(X,Y)
\end{equation*}
and since 
\begin{equation*}
g\left( JX,Y\right) =g(JX_{0},Y_{0})-u\wedge v(X,Y)
\end{equation*}%
we obtained 
\begin{equation*}
d\sigma \left( X,Y\right) =2g\left( JX_{0},Y_{0}\right) +g\left( \left(
\nabla _{U}J\right) GX_{0},Y_{0}\right) +d\sigma (U,V)u\wedge v(X,Y).
\end{equation*}
\end{proof}

\begin{theorem}
Let $M$ be a complex contact metric manifold. Then $M$ is normal if and only
if\bigskip 
\begin{eqnarray}
(\nabla _{X}G)Y &=&\sigma (X)HY-2v(X)JY-u\left( Y\right) X \\
&&-v(Y)JX+v(X)\left( 2JY_{0}-\left( \nabla _{U}J\right) GY_{0}\right)  \notag
\\
&&+g(X,Y)U+g(JX,Y)V  \notag \\
&&-d\sigma (U,V)v(X)\left( u(Y)V-v(Y)U\right)  \notag
\end{eqnarray}%
\begin{eqnarray}
(\nabla _{X}H)Y &=&-\sigma (X)GY+2u(X)JY+u(Y)JX \\
&&-v(Y)X+u(X)\left( -2JY_{0}-\left( \nabla _{U}J\right) GY_{0}\right)  \notag
\\
&&-g(JX,Y)U+g(X,Y)V  \notag \\
&&+d\sigma (U,V)u(X)\left( u(Y)V-v(Y)U\right) .  \notag
\end{eqnarray}
\end{theorem}

\begin{proof}
Suppose that $M$ is a normal complex contact metric manifold. \ Then from
(2.4) and (2.5) we have%
\begin{eqnarray*}
g\left( (\nabla _{X}G)Y,Z\right) &=&g\left( \sigma (X)HY-2v(X)JY-u\left(
Y\right) X-v\left( Y\right) JX\right. \\
&&\left. +g\left( X,Y\right) U+g\left( X,JY\right) V,Z\right) +v\left(
X\right) d\sigma \left( GZ,GY\right) .
\end{eqnarray*}%
and%
\begin{eqnarray*}
g\left( (\nabla _{X}H)Y,Z\right) &=&g\left( -\sigma (X)GY+2u(X)JY+u\left(
Y\right) JX-v\left( Y\right) X\right. \\
&&\left. -g\left( JX,Y\right) U+g\left( X,Y\right) V,Z\right) -u(X)d\sigma
\left( HZ,HY\right) .
\end{eqnarray*}%
Since $u\wedge v(Y,Z)=g\left( u(Y\right) V-v(Y)U,Z)$ and from (2.9) we get%
\begin{eqnarray*}
g\left( (\nabla _{X}G)Y,Z\right) &=&g\left( \sigma (X)HY-2v(X)JY-u\left(
Y\right) X-v\left( Y\right) JX\right. \\
&&\left. +g\left( X,Y\right) U+g\left( X,JY\right) V,Z\right) \\
&&v\left( X\right) \left[ d\sigma (Y,Z)-2d\sigma \left( U,V\right) g\left(
u(Y)V-v\left( Y\right) U,Z\right) \right]
\end{eqnarray*}%
and%
\begin{eqnarray*}
g\left( (\nabla _{X}H)Y,Z\right) &=&g\left( -\sigma (X)GY+2u(X)JY-u\left(
Y\right) JX-v\left( Y\right) X\right. \\
&&\left. -g\left( JX,Y\right) U+g\left( X,Y\right) V,Z\right) \\
&&+u(X)\left[ d\sigma (Y,Z)-2d\sigma \left( U,V\right) g\left( u(Y)V-v\left(
Y\right) U,Z\right) \right] .
\end{eqnarray*}%
From (4.11) we can write%
\begin{equation*}
d\sigma \left( Y,Z\right) =g(2JY_{0}+\left( \bigtriangledown _{U}J\right)
GY_{0}+d\sigma \left( U,V\right) \left( u(Y)V-v\left( Y\right) U),Z\right) .
\end{equation*}%
By using this equation we obtain (4.12) and\ (4.13).

\qquad Conversely suppose that (4.12) and (4.13)\ hold. For arbitrary vector
field $X$\ and from (2.2), (2.3)\ we have 
\begin{equation*}
S(X,U)=(\nabla _{GX}G)U-G(\nabla _{X}G)U+G(\nabla _{U}G)X-\sigma \left(
U\right) GHX
\end{equation*}%
\begin{equation*}
T(X,V)=(\nabla _{HX}H)V-H(\nabla _{X}H)V+H(\nabla _{V}H)X-\sigma \left(
V\right) GHX
\end{equation*}%
From (4.12) and (4.13)\ we get $S(X,U)=T(X,V)=0.$

Now let $X$ and $Y$ be two horizontal vector fields. Since $u\left( X\right)
=u(Y)=v\left( X\right) =v(Y)=0$ we have \ 
\begin{eqnarray*}
S(X,Y) &=&(\nabla _{GX}G)Y-(\nabla _{GY}G)X-G(\nabla _{X}G)Y+G(\nabla _{Y}G)X
\\
&&+2g(X,GY)U-2g(X,HY)V+\sigma (GY)HX \\
&&-\sigma (GX)HY+\sigma (X)GHY-\sigma (Y)GHX
\end{eqnarray*}%
and 
\begin{eqnarray*}
T(X,Y) &=&(\nabla _{HX}H)Y-(\nabla _{HY}H)X-H(\nabla _{X}H)Y+H(\nabla _{Y}H)X
\\
&&-2g(X,GY)U+2g(X,HY)V+\sigma (HX)GY \\
&&-\sigma (HY)GX+\sigma (X)GHY-\sigma (Y)GHX.
\end{eqnarray*}%
By aplying (4.12) and (4.13) we get \ 
\begin{eqnarray*}
S(X,Y) &=&\sigma (GX)HY-\sigma (GY)HX-2g(X,GY)U \\
&&+2g(X,HY)V-\sigma (X)GHY+\sigma (Y)GHX \\
&&+2g(X,GY)U-2g(X,HY)V+\sigma (GY)HX \\
&&-\sigma (GX)HY+\sigma (X)GHY-\sigma (Y)GHX \\
&=&0
\end{eqnarray*}%
and%
\begin{eqnarray*}
T(X,Y) &=&-\sigma (HX)GY+\sigma (HY)GX+\sigma (X)HGY \\
&&-\sigma (Y)HGX+2g(X,GY)U-2g(X,HY)V \\
&&-2g(X,GY)U+2g(X,HY)V+\sigma (HX)GY \\
&&-\sigma (HY)GX+\sigma (X)GHY-\sigma (Y)GHX \\
&=&0
\end{eqnarray*}%
Therefore M is normal.
\end{proof}

Using (2.6), (4.11), (4.12)\ and (4.13)\ we obtained following corollary.

\begin{corollary}
Let $M$ be a normal complex contact metric manifold and $X,Y$ be two
arbitrary vector fields on $M$. Then we have 
\begin{eqnarray*}
(\nabla _{X}J)Y &=&-2u\left( X\right) HY+2v(X)GY+u(X)\left( 2HY_{0}+\left(
\nabla _{U}J\right) Y_{0}\right) \\
&&+v(X)\left( -2GY_{0}+\left( \nabla _{U}J\right) JY_{0}\right) .
\end{eqnarray*}
\end{corollary}

\section{\protect\bigskip \textbf{Some Properties of Ricci Curvature for
Normal Complex Contact Metric Manifolds}}

\begin{theorem}
Let $M$ be a normal complex contact metric manifold and $X,Y$\ be horizontal
vector fields on $M$. Then we have 
\begin{equation}
\rho \left( GX,GY\right) =\rho \left( HX,HY\right) =\rho \left( X,Y\right)
\end{equation}%
\begin{equation}
\rho \left( GX,Y\right) =-\rho \left( X,GY\right) ,~\text{ }\rho \left(
HX,Y\right) =-\rho \left( X,HY\right) .
\end{equation}
\end{theorem}

\begin{proof}
Let us choose a local orthonormal basis of the form 
\begin{equation*}
\left\{ X_{i},GX_{i},HX_{i},JX_{i},U,V:1\leq i\leq n\right\} .
\end{equation*}%
Then the Ricci tensor has the form%
\begin{eqnarray}
\rho (X,Y) &=&\underset{i=1}{\overset{n}{\sum }}%
[g(R(X_{i,}X)Y,X_{i})+g(R(GX_{i},X)Y,GX_{i}) \\
&&+g(R(HX_{i},X)Y,HX_{i})+g(R(JX_{i},X)Y,JX_{i})]  \notag \\
&&+g(R(U,X)Y,U)+g(R(V,X)Y,V).  \notag
\end{eqnarray}%
Therefore from (5.3)\ we can write 
\begin{eqnarray}
\rho (GX,GY) &=&\underset{i=1}{\overset{n}{\sum }}[g(R(X_{i,}GX)GY,X_{i}) \\
&&+g(R(GX_{i},GX)GY,GX_{i})+g(R(HX_{i},GX)GY,HX_{i})  \notag \\
&&+g(R(JX_{i},GX)GY,JX_{i})]+g(R(U,GX)GY,U)  \notag \\
&&+g(R(V,GX)GY,V)  \notag
\end{eqnarray}%
From (4.1) we have%
\begin{equation*}
g(R(X_{i,}GX)GY,X_{i})=g(R(GX_{i,}GGX)GGY,GX_{i})=(g(R(GX_{i,}X)Y,GX_{i})
\end{equation*}%
\begin{equation*}
g(R(GX_{i},GX)GY,GX_{i})=g(R(X_{i},X)Y,X_{i})
\end{equation*}%
\begin{equation*}
g(R(HX_{i},GX)GY,HX_{i})=g(R(GJX_{i},GX)GY,GJX_{i})=g(R(JX_{i},X)Y,JX_{i})
\end{equation*}%
\begin{equation*}
g(R(JX_{i},GX)GY,JX_{i})=g(R(-GHX_{i},GX)GY,-GHX_{i})=g(R(HX_{i},X)Y,HX_{i}).
\end{equation*}%
From (2.12) \ we have 
\begin{eqnarray*}
g(R(U,GX)GY,U) &=&g(R(GY,U)U,GX) \\
&=&g(GY,GX) \\
&=&g(Y,X) \\
&=&g(R(Y,U)U,X) \\
&=&g(R(U,X)Y,U)
\end{eqnarray*}%
and 
\begin{equation*}
g(R(V,GX)GY,V)=g(X,Y)=g(R(V,X)Y,V).
\end{equation*}%
Using these equations in (5.3) we showed that $\rho (GX,GY)=\rho (X,Y).$
Similarly from (5.1) we have 
\begin{eqnarray}
\begin{array}{cc}
& 
\end{array}%
\rho (HX,HY) &=&\underset{i=1}{\overset{n}{\sum }}[g(R(X_{i,}HX)HY,X_{i}) \\
&&+g(R(GX_{i},HX)HY,GX_{i})+g(R(HX_{i},HX)HY,HX_{i})  \notag \\
&&+g(R(JX_{i},HX)HY,JX_{i})]+g(R(U,HX)HY,U)  \notag \\
&&+g(R(V,HX)HY,V).  \notag
\end{eqnarray}%
and from (4.1) we get%
\begin{equation*}
g(R(X_{i},HX)HY,X_{i})=g(R(HX_{i,}HHX)HHY,HX_{i})=(g(R(HX_{i,}X)Y,HX_{i})
\end{equation*}%
\begin{equation*}
g(R(HX_{i},HX)HY,HX_{i})=g(R(X_{i},X)Y,X_{i})
\end{equation*}%
\begin{equation*}
g(R(GX_{i},HX)HY,GX_{i})=g(R(-HJX_{i},HX)HY,-HJX_{i})=g(R(JX_{i},X)Y,JX_{i})
\end{equation*}%
\begin{equation*}
g(R(JX_{i},HX)HY,JX_{i})=g(R(HGX_{i},HX)HY,HGX_{i})=g(R(GX_{i},X)Y,GX_{i}).
\end{equation*}%
From (2.12) we have 
\begin{eqnarray*}
g(R(U,HX)HY,U) &=&g(R(U,X)Y,U) \\
g(R(V,HX)HY,V) &=&g(R(V,X)Y,V).
\end{eqnarray*}%
So using these equations in (5.3)\ we get $\rho (HX,HY)=\rho (X,Y).$ From
(5.1) we have%
\begin{eqnarray*}
\rho (GX,Y) &=&\underset{i=1}{\overset{n}{\sum }}%
[g(R(X_{i,}GX)Y,X_{i})+g(R(GX_{i},GX)Y,GX_{i}) \\
&&+g(R(HX_{i},GX)Y,HX_{i})+g(R(JX_{i},GX)Y,JX_{i})] \\
&&+g(R(U,GX)Y,U)+g(R(V,GX)Y,V)
\end{eqnarray*}%
and%
\begin{eqnarray*}
\rho (HX,Y) &=&\underset{i=1}{\overset{n}{\sum }}%
[g(R(X_{i,}HX)Y,X_{i})+g(R(GX_{i},HX)Y,GX_{i}) \\
&&+g(R(HX_{i},HX)Y,HX_{i})+g(R(JX_{i},HX)Y,JX_{i})] \\
&&+g(R(U,HX)Y,U)+g(R(V,HX)Y,V)
\end{eqnarray*}%
From (4.1) and (2.12)\ we obtain

\begin{equation*}
\rho \left( GX,Y\right) =-\rho \left( X,GY\right) \text{ \ \ \ and \ \ }\rho
\left( HX,Y\right) =-\rho \left( X,HY\right) .
\end{equation*}
\end{proof}

\begin{theorem}
Let $M$ be a normal complex contact metric manifold. For any $X$ horizontal
vector filed and $U,V=-JU$\ vertical vector fields on $M$\ the Ricci
curvature tensor satisfies$\bigskip $%
\begin{equation}
\rho (X,U)=\rho (X,V)=0
\end{equation}%
\begin{equation}
\rho (U,U)=\rho (V,V)=4n-2d\sigma (U,V),\text{ \ }\rho (U,V)=0.
\end{equation}
\end{theorem}

\begin{proof}
For any $X$ horizontal vector field and $U,V=-JU$\ vertical vector fields on 
$M$ from (5.3) we can write%
\begin{eqnarray}
\rho (X,U) &=&\underset{i=1}{\overset{n}{\sum }}[g(R(X_{i,}X)U,X_{i}) \\
&&+g(R(GX_{i},X)U,GX_{i})+g(R(HX_{i},X)U,HX_{i})  \notag \\
&&+g(R(JX_{i},X)U,JX_{i})]+g(R(U,X)U,U)  \notag \\
&&+g(R(V,X)U,V)  \notag
\end{eqnarray}%
\begin{eqnarray}
\rho (X,V) &=&\underset{i=1}{\overset{n}{\sum }}[g(R(X_{i,}X)V,X_{i}) \\
&&+g(R(GX_{i},X)V,GX_{i})+g(R(HX_{i},X)V,HX_{i})  \notag \\
&&+g(R(JX_{i},X)V,JX_{i})]+g(R(U,X)V,U)  \notag \\
&&+g(R(V,X)V,V)  \notag
\end{eqnarray}%
Since $X_{i,}$ $X$ are horizontal vector fields on $M$ from (2.13) and (2.14)%
\begin{eqnarray*}
g(R(X_{i,}X)U,X_{i}) &=&g(2(g(X_{i,},JX)+d\sigma (X_{i,},X))V,X_{i,}) \\
&=&2(g(X_{i,},JX)+d\sigma (X_{i,},X))g(V,X_{i})
\end{eqnarray*}%
and%
\begin{eqnarray*}
g(R(X_{i,}X)V,X_{i}) &=&g(-2(g(X_{i,},JX)+d\sigma (X_{i,},X))U,X_{i,}) \\
&=&-2(g(X_{i,},JX)+d\sigma (X_{i,},X)g(U,X_{i})
\end{eqnarray*}%
Since $g(U,X_{i})=g(V,X_{i})=0$ we get 
\begin{equation*}
g(R(X_{i,}X)U,X_{i})=0\text{ \ and }g(R(X_{i,}X)V,X_{i})=0
\end{equation*}%
By the same way and since $%
g(U,X_{i})=g(U,GX_{i})=g(U,HX_{i})=g(U,JX_{i})=g(V,X_{i})=g(V,GX_{i})=g(V,HX_{i})=g(V,JX_{i})=0 
$ we have 
\begin{equation*}
g(R(GX_{i},X)U,GX_{i})=g(R(HX_{i},X)U,HX_{i})=g(R(JX_{i},X)U,JX_{i})=0
\end{equation*}%
and%
\begin{equation*}
g(R(GX_{i},X)V,GX_{i})=g(R(HX_{i},X)V,HX_{i})=g(R(JX_{i},X)V,JX_{i})=0.
\end{equation*}%
On the other \ hand using (2.12),(2.16) and (2.17) we have 
\begin{equation*}
g(R(U,X)U,U)=-g(R(X,U)U,U)=-g(X,U)=0,
\end{equation*}%
\begin{equation*}
g(R(V,X)V,V)=-g(R(X,V)V,V)=-g(X,V)=0,
\end{equation*}%
\begin{eqnarray*}
g(R(V,X)U,V). &=&-g(R(X,V)U,V) \\
&=&-g(-\sigma (V)HX+(\nabla _{V}G)X+JX,V) \\
&=&0,
\end{eqnarray*}%
\begin{eqnarray*}
g(R(U,X)V,U) &=&-g(R(X,U)V,U) \\
&=&-g(\sigma (U)GX+(\bigtriangledown _{U}H)X-JX,U) \\
&=&0.
\end{eqnarray*}%
Using these equations in (5.8) and (5.9) we get (5.6). On the other hand 
\begin{eqnarray*}
\rho (U,U) &=&\underset{i=1}{\overset{n}{\sum }}%
[g(R(X_{i,}U)U,X_{i})+g(R(GX_{i},U)U,GX_{i})+g(R(HX_{i},U)U,HX_{i}) \\
&&+g(R(JX_{i},U)U,JX_{i})]+g(R(U,U)U,U)+g(R(V,U)U,V)
\end{eqnarray*}%
and%
\begin{eqnarray*}
\rho (V,V) &=&\underset{i=1}{\overset{n}{\sum }}%
[g(R(X_{i,}V)V,X_{i})+g(R(GX_{i},V)V,GX_{i})+g(R(HX_{i},V)V,HX_{i}) \\
&&+g(R(JX_{i},V)V,JX_{i})]+g(R(U,V)V,U)+g(R(V,V)V,V)
\end{eqnarray*}%
from (2.11), (2.12) and by direct computation we get (5.7).
\end{proof}

\begin{corollary}
For arbitrary $X$ vector field on normal complex contact metric manifold $M$
we have 
\begin{eqnarray}
\rho (X,U) &=&\left( 4n-2d\sigma (U,V)\right) u(X) \\
\rho \left( X,V\right) &=&\left( 4n-2d\sigma (U,V)\right) v(X).  \notag
\end{eqnarray}

\begin{proof}
For a $X=X_{0}+u(X)U+v(X)V$ vector field we can write 
\begin{eqnarray*}
\rho (X,U) &=&\rho (X_{0}+u(X)U+v(X)V,U) \\
&=&\rho \left( X_{0},U\right) +u(X)\rho \left( U,U\right) +v(X)\rho \left(
V,U\right)
\end{eqnarray*}%
and 
\begin{eqnarray*}
\rho (X,V) &=&\rho (X_{0}+u(X)U+v(X)V,V) \\
&=&\rho \left( X_{0},V\right) +u(X)\rho \left( U,V\right) +v(X)\rho \left(
V,V\right)
\end{eqnarray*}%
from (5.6) and (5.7) we finished the proof.
\end{proof}
\end{corollary}

\begin{corollary}
Let $M$ be a normal complex contact metric manifold and $X,Y\ $be two
arbitrary vector fields on $M$. Then Ricci curvature tensor satisfies%
\begin{equation}
\rho (X,Y)=\rho (X_{0},Y_{0})+\left( 4n-2d\sigma (U,V)\right) \left(
u(X)u(Y)+v(X)v(Y)\right)
\end{equation}%
wehere $X=X_{0}+u(X)U+v(X)V,~Y=Y_{0}+u(Y)U+v(Y)V$ \ where $X_{0}$ and $Y_{0} 
$ are in $\mathcal{H}.$
\end{corollary}

\begin{proof}
For $X=X_{0}+u(X)U+v(X)V,~Y=Y_{0}+u(Y)U+v(Y)V$ vector fields on $M$ we have 
\begin{eqnarray*}
\rho (X,Y) &=&\rho (X_{0}+u(X)U+v(X)V,Y) \\
&=&\rho (X_{0},Y)+u\left( X\right) \rho (U,Y)+v(X)\rho (V,Y)
\end{eqnarray*}%
and from (5.10) 
\begin{eqnarray*}
\rho (X,Y) &=&\rho (X_{0},Y_{0}+u(Y)U+v(Y)V)+u(X)u(Y)\left( 4n-2d\sigma
(U,V)\right) \\
&&+v(X)v(Y)\left( 4n-2d\sigma (U,V)\right) .
\end{eqnarray*}

By using (5.6) and (5.7) we obtain\ 
\begin{equation*}
\rho (X,Y)=\rho (X_{0},Y_{0})+\left( 4n-2d\sigma (U,V)\right) \left(
u(X)u(Y)+v(x)v(Y)\right) .
\end{equation*}
\end{proof}

\begin{corollary}
Let $M$ be a normal complex contact metric manifold and $X,Y$\ be two
arbitrary vector fields on $M$. Then Ricci curvature tensor satisfies 
\begin{eqnarray*}
\rho (X,Y) &=&\rho (GX,GY)+\left( 4n-2d\sigma (U,V)\right) \left(
u(X)u(Y)+v(X)v(Y)\right) , \\
\rho (X,Y) &=&\rho (HX,HY)+\left( 4n-2d\sigma (U,V)\right) \left(
u(X)u(Y)+v(X)v(Y)\right) .
\end{eqnarray*}
\end{corollary}

\begin{proof}
If \ $X=X_{0}+u(X)U+v(X)V,~Y=Y_{0}+u(Y)U+v(Y)V$ since $GX=GX_{0}$, $%
GY=GY_{0} $ and $HX=HX_{0},HY=HY_{0}$ then 
\begin{equation*}
\rho (GX,GY)=\rho (GX_{0},GY_{0})\text{ and }\rho (HX,HY)=\rho
(HX_{0},HY_{0})\text{ .}
\end{equation*}%
By using (5.11) we have 
\begin{eqnarray*}
\rho (X,Y) &=&\rho (GX,GY)+\left( 4n-2d\sigma (U,V)\right) \left(
u(X)u(Y)+v(X)v(Y)\right) , \\
\rho (X,Y) &=&\rho (HX,HY)+\left( 4n-2d\sigma (U,V)\right) \left(
u(X)u(Y)+v(X)v(Y)\right) .
\end{eqnarray*}
\end{proof}

\begin{corollary}
On a normal complex contact metric manifold $M,$ for $Q$ Ricci operator we
have 
\begin{equation*}
QG=GQ,\text{ }~QH=HQ.
\end{equation*}
\end{corollary}

\begin{proof}
For $X,Y$ vector fields on M we have 
\begin{equation*}
\rho \left( GX,Y\right) =g(GX,QY)=-g(X,GQY)
\end{equation*}%
and%
\begin{equation*}
\rho \left( X,GY\right) =g(X,QGY)
\end{equation*}%
from (5.2) we get 
\begin{equation*}
g(X,GQY)=g(X,QGY)
\end{equation*}%
and $QG=GQ.~$By same way one can show that $QH=HQ.$
\end{proof}

\section{\protect\bigskip \textbf{Example of a Normal Complex Contact Metric
Manifolds}}

The complex Heisenberg group is the closed subgroup $H_{{%
\mathbb{C}
}}$ of $GL(3,%
\mathbb{C}
)$ given by 
\begin{equation*}
H_{%
\mathbb{C}
}=\left\{ \left( 
\begin{array}{ccc}
1 & b_{12} & b_{13} \\ 
0 & 1 & b_{23} \\ 
0 & 0 & 1%
\end{array}%
\right) :b_{12},b_{13},b_{23}\in 
\mathbb{C}
\right\} \simeq 
\mathbb{C}
^{3}
\end{equation*}

Blair defined the following complex contact metric structure on $H_{%
\mathbb{C}
}$ in \cite{BL98} Let $z_{1},z_{2},z_{3}$be the coordinates on $H_{%
\mathbb{C}
}\simeq 
\mathbb{C}
^{3}$ defined by $z_{1}(B)=b_{23},z_{2}(B)=b_{12},z_{3}(B)=b_{13}$ for $B$
in $H_{%
\mathbb{C}
}.$ Here $H_{%
\mathbb{C}
}\simeq 
\mathbb{C}
^{3}$ and $\theta =\frac{1}{2}\left( dz_{3}-z_{2}dz_{1}\right) $ is global,
so the structure tensors may be taken globally. With J denoting the standard
almost complex structure on $%
\mathbb{C}
^{3}$, we may give a complex almost contact structure to $H_{%
\mathbb{C}
}$ as follows. Since $\theta $ is holomorphic, set\ $\theta =u+iv$ , $%
v=u\circ J;$ also set $4\frac{\partial }{\partial z_{3}}=U+iV$. Then $%
u(X)=g(U,X)$ and $v(X)=g(V,X).$ Since we will work in real coordinates, $G$
and $H$ are given by 
\begin{equation*}
G=\left[ 
\begin{array}{cccccc}
0 & 0 & 1 & 0 & 0 & 0 \\ 
0 & 0 & 0 & 1 & 0 & 0 \\ 
1 & 0 & 0 & 0 & 0 & 0 \\ 
0 & 1 & 0 & 0 & 0 & 0 \\ 
0 & 0 & x_{2} & y_{2} & 0 & 0 \\ 
0 & 0 & y_{2} & x_{2} & 0 & 0%
\end{array}%
\right] 
\end{equation*}%
\begin{equation*}
H=\left[ 
\begin{array}{cccccc}
0 & 0 & 0 & 1 & 0 & 0 \\ 
0 & 0 & 1 & 0 & 0 & 0 \\ 
0 & 1 & 0 & 0 & 0 & 0 \\ 
1 & 0 & 0 & 0 & 0 & 0 \\ 
0 & 0 & y_{2} & x_{2} & 0 & 0 \\ 
0 & 0 & x_{2} & y_{2} & 0 & 0%
\end{array}%
\right] 
\end{equation*}%
Then relative to the coordinates $\left(
x_{1},y_{1},x_{2},y_{2},x_{3},y_{3}\right) $ the Hermitian metric (matrix)%
\begin{equation*}
g=\frac{1}{4}\left[ 
\begin{array}{cccccc}
1+x_{2}^{2}+y_{2}^{2} & 0 & 0 & 0 & x_{2} & y_{2} \\ 
0 & 1+x_{2}^{2}+y_{2}^{2} & 0 & 0 & y_{2} & x_{2} \\ 
0 & 0 & 1 & 0 & 0 & 0 \\ 
0 & 0 & 0 & 1 & 0 & 0 \\ 
x_{2} & y_{2} & 0 & 0 & 1 & 0 \\ 
y_{2} & x_{2} & 0 & 0 & 0 & 1%
\end{array}%
\right] 
\end{equation*}%
in addition $\left\{ e_{1};e_{1}^{\ast };e_{2};e_{2}^{\ast
};e_{3};e_{3}^{\ast }\right\} $ is an orthonormal basis where%
\begin{eqnarray}
e_{1} &=&2\left( \frac{\partial }{\partial x_{1}}+x_{2}\frac{\partial }{%
\partial x_{3}}+y_{2}\frac{\partial }{\partial y_{3}}\right) ~,~ \\
e_{1}^{\ast } &=&2\left( \frac{\partial }{\partial y_{1}}+y_{2}\frac{%
\partial }{\partial x_{3}}+x_{2}\frac{\partial }{\partial y_{3}}\right)  
\notag \\
e_{2} &=&2\frac{\partial }{\partial x_{2}}\text{ \ \ \ , }e_{2}^{\ast }=2%
\frac{\partial }{\partial y_{2}}\text{\ ,~}  \notag \\
e_{3} &=&U=2\frac{\partial }{\partial x_{3}}\text{ \ , \ }e_{3}^{\ast }=V=2%
\frac{\partial }{\partial y_{3}}.  \notag
\end{eqnarray}%
Furthermore we have \cite{VAN2004}%
\begin{eqnarray*}
Ge_{1} &=&-e_{2},~Ge_{1}^{\ast }=e_{2}^{\ast },~Ge_{2}=e_{1},~Ge_{2}^{\ast
}=-e_{1}^{\ast } \\
He_{1} &=&-e_{2}^{\ast },~He_{1}^{\ast }=-e_{2},~He_{2}=e_{1}^{\ast
},~He_{2}^{\ast }=-e_{1} \\
Je_{1} &=&-e_{1}^{\ast },~Je_{1}^{\ast }=e_{1},~Je_{2}=-e_{2}^{\ast
},~Je_{2}^{\ast }=-e_{2}.
\end{eqnarray*}%
Let $\bigtriangledown $ be the Levi-Civita connection with respect to metric 
$g.$ Then from (6.1) we have 
\begin{equation}
\left[ e_{1},e_{2}\right] =-2e_{3},~\left[ e_{1},e_{2}^{\ast }\right]
=-2e_{3}^{\ast },~\left[ e_{1}^{\ast },e_{2}\right] =-2e_{3}^{\ast },~\left[
e_{1}^{\ast },e_{2}^{\ast }\right] =2e_{3}
\end{equation}%
and the other Lie brackets are zero \cite{VAN2004}. In addition we have%
\begin{equation*}
2g(\bigtriangledown _{e_{i}}e_{j},e_{k})=g\left[ e_{i},e_{j}\right]
,e_{k}+g\left( \left[ e_{k},e_{i}\right] ,e_{j}\right) -g\left( \left[
e_{j},e_{k}\right] ,e_{i}\right) 
\end{equation*}%
and from that we obtain 
\begin{equation}
\bigtriangledown _{e_{j}}e_{j}=\bigtriangledown
_{e_{j}}e_{j}=\bigtriangledown _{e_{j}}e_{j^{\ast }}=\bigtriangledown
_{e_{j}^{\ast }}e_{j}^{\ast }=0,
\end{equation}%
where $j=1,2,3.$ From (6.2) and (6.3) we need only list following 
\begin{eqnarray*}
\bigtriangledown _{e_{2}}e_{3} &=&\bigtriangledown _{e_{2}^{\ast
}}e_{3}^{\ast }=-e_{1}~~~~~~\bigtriangledown _{e_{2}^{\ast
}}e_{3}=-\bigtriangledown _{e_{2}}e_{3}^{\ast }=e_{1}^{\ast } \\
\bigtriangledown _{e_{1}}e_{3} &=&\bigtriangledown _{e_{1}^{\ast
}}e_{3}^{\ast }=e_{2}~~~~~~~~\bigtriangledown _{e_{1}}e_{3}^{\ast
}=-\bigtriangledown _{e_{1}^{\ast }}e_{3}=e_{2}^{\ast } \\
-\bigtriangledown _{e_{1}}e_{2} &=&\bigtriangledown _{e_{1}^{\ast
}}e_{2}^{\ast }=e_{3}~~~~~~\ ~~\bigtriangledown _{e_{1}}e_{2}^{\ast
}=\bigtriangledown _{e_{1}^{\ast }}e_{2}=-e_{3}^{\ast }.
\end{eqnarray*}%
Now let 
\begin{equation*}
\Gamma =\left\{ \left. \left( 
\begin{array}{ccc}
1 & \gamma _{2} & \gamma _{3} \\ 
0 & 1 & \gamma _{1} \\ 
0 & 0 & 1%
\end{array}%
\right) \right\vert \gamma _{k}=m_{k}+in_{k},~m_{k},~n_{k}\in 
\mathbb{Z}
\right\} 
\end{equation*}%
$\Gamma $ is subgroup of $H_{%
\mathbb{C}
}\simeq 
\mathbb{C}
^{3},$ the 1-form $dz_{3}-z_{2}dz_{1}$ is invariant under the action on $%
\Gamma $ and with $\xi =U\wedge V$, hence the quotient $\left. H_{%
\mathbb{C}
}\right/ \Gamma $ is a compact complex contact manifold with a global
complex contact form. $\left. H_{%
\mathbb{C}
}\right/ \Gamma $ is known the \textit{Iwasawa manifold.}

It is known that with the help of the above results , it can be easily
verified that%
\begin{equation*}
\begin{tabular}{cccc}
\begin{tabular}{c}
$R(e_{1},e_{1}^{\ast })e_{1}=0~~~~$ \\ 
$R(e_{1},e_{1}^{\ast })e_{1}^{\ast }=0~~~~\ $ \\ 
$R(e_{1},e_{1}^{\ast })e_{2}=-2e_{2}^{\ast }$ \\ 
$R(e_{1},e_{1}^{\ast })e_{2}^{\ast }=2e_{2}~~~$%
\end{tabular}
& 
\begin{tabular}{l}
\begin{tabular}{c}
$R(e_{1},e_{2})e_{1}=3e_{2}$ \\ 
$R(e_{1},e_{2})e_{1}^{\ast }=-e_{2}^{\ast }$ \\ 
$R(e_{1},e_{2})e_{2}=-3e_{1}$ \\ 
$R(e_{1},e_{2})e_{2}^{\ast }=e_{1}^{\ast }$%
\end{tabular}%
\end{tabular}
& 
\begin{tabular}{c}
$R(e_{1},e_{2}^{\ast })e_{1}=3e_{2}$ \\ 
$R(e_{1},e_{2}^{\ast })e_{1}^{\ast }=0$ \\ 
$R(e_{1},e_{2}^{\ast })e_{2}=-e_{2}^{\ast }$ \\ 
$R(e_{1},e_{2}^{\ast })e_{2}^{\ast }=-3e_{1}$%
\end{tabular}
& 
\begin{tabular}{c}
$R(e_{1}^{\ast },e_{2})e_{1}=e_{2}^{\ast }$ \\ 
$R(e_{1}^{\ast },e_{2})e_{1}^{\ast }=2e_{2}$ \\ 
$R(e_{1}^{\ast },e_{2})e_{2}=-3e_{1}^{\ast }$ \\ 
$R(e_{1}^{\ast },e_{2})e_{2}^{\ast }=3e_{1}$%
\end{tabular}%
\end{tabular}%
\end{equation*}%
\begin{equation*}
\begin{tabular}{l}
\begin{tabular}{c}
$R(e_{1}^{\ast },e_{2}^{\ast })e_{1}=-e_{2}$ \\ 
$R(e_{1}^{\ast },e_{2}^{\ast })e_{1}^{\ast }=e_{2}^{\ast }$ \\ 
$R(e_{1}^{\ast },e_{2}^{\ast })e_{2}=e_{1}$ \\ 
$R(e_{1}^{\ast },e_{2}^{\ast })e_{2}^{\ast }=-3e_{1}^{\ast }$%
\end{tabular}%
\end{tabular}%
\begin{tabular}{c}
$R(e_{2},e_{2}^{\ast })e_{1}=-2e_{1}^{\ast }$ \\ 
$R(e_{2},e_{2}^{\ast })e_{1}^{\ast }=2e_{1}$ \\ 
$R(e_{2},e_{2}^{\ast })e_{2}=0$ \\ 
$R(e_{2},e_{2}^{\ast })e_{2}^{\ast }=0~.$%
\end{tabular}%
\end{equation*}

From (2.8) and since $\sigma =0$ \cite{KB2000} \ $R(e_{3},e_{3}^{\ast
})e_{3}^{\ast }=0$ and we have $R(X,U)U=X$ and $R(X,V)V=X$ for $X\in 
\mathcal{H}$. Similarly from $\sigma =0,$ (2.15) and (2.16) we get 
\begin{equation*}
\begin{tabular}{l}
\begin{tabular}{c}
$R(e_{1},e_{3})e_{3}^{\ast }=-e_{1}^{\ast }$ \\ 
$R(e_{1}^{\ast },e_{3}^{\ast })e_{3}=3e_{1}$ \\ 
$R(e_{2},e_{3})e_{3}^{\ast }=-e_{2}^{\ast }$ \\ 
$R(e_{2}^{\ast },e_{3}^{\ast })e_{3}=e_{2}$%
\end{tabular}%
\begin{tabular}{l}
\begin{tabular}{c}
$R(e_{1},e_{3}^{\ast })e_{3}=e_{1}^{\ast }$ \\ 
$R(e_{1}^{\ast },e_{3}^{\ast })e_{3}=-e_{1}$ \\ 
$R(e_{2},e_{3}^{\ast })e_{3}=e_{2}^{\ast }$ \\ 
$R(e_{2}^{\ast },e_{3}^{\ast })e_{3}=-3e_{2}$%
\end{tabular}%
\end{tabular}%
\end{tabular}%
\end{equation*}%
and from (2.13),(2.14) we get%
\begin{equation*}
\begin{tabular}{l}
\begin{tabular}{c}
$R(e_{1},e_{1}^{\ast })e_{3}=R(e_{1},e_{2})e_{3}=R(e_{1},e_{2}^{\ast
})e_{3}=0$ \\ 
$R(e_{1}^{\ast },e_{2})e_{3}=R(e_{1}^{\ast },e_{2}^{\ast
})e_{3}=R(e_{2},e_{2}^{\ast })e_{3}=0~~~$%
\end{tabular}%
\end{tabular}%
\end{equation*}%
and%
\begin{equation*}
\begin{tabular}{c}
$R(e_{1},e_{1}^{\ast })e_{3}^{\ast }=R(e_{1},e_{2})e_{3}^{\ast
}=R(e_{1},e_{2}^{\ast })e_{3}^{\ast }=0$ \\ 
$R(e_{1}^{\ast },e_{2})e_{3}^{\ast }=R(e_{1}^{\ast },e_{2}^{\ast
})e_{3}^{\ast }=R(e_{2},e_{2}^{\ast })e_{3}^{\ast }=0~~~.$%
\end{tabular}%
\end{equation*}%
Using these equations and from (5.2)\ Ricci curvature is%
\begin{eqnarray*}
\rho \left( e_{i},e_{i}\right) &=&\rho \left( e_{i}^{\ast },e_{i}^{\ast
}\right) =4\text{, }i=1,2,3 \\
\rho \left( e_{i},e_{j}\right) &=&\rho \left( e_{i}^{\ast },e_{j}^{\ast
}\right) =0,~j=1,2,3.
\end{eqnarray*}%
By direct computation the scalar curvature of Iwasava manifold is $\tau =24.$
Furthermore from curvature equalities the sectional curvature is 
\begin{eqnarray*}
k(e_{1},e_{3}) &=&k(e_{1}^{\ast },e_{3})=k(e_{2},e_{3})=k(e_{2}^{\ast
},e_{3})=1, \\
k(e_{1},e_{3}^{\ast }) &=&k(e_{1}^{\ast },e_{3}^{\ast })=k(e_{2},e_{3}^{\ast
})=k(e_{2}^{\ast },e_{3}^{\ast })=1
\end{eqnarray*}%
and since $\sigma =0$ 
\begin{equation*}
k(e_{3},e_{3}^{\ast })=0.
\end{equation*}%
In addition can be easily verified that%
\begin{eqnarray*}
k(e_{1},e_{1}^{\ast }) &=&k(e_{1},e_{2}^{\ast })=k(e_{1}^{\ast
},e_{2})=k(e_{2},e_{2}^{\ast })=0 \\
k(e_{1},e_{2}) &=&3\text{ and }k(e_{1}^{\ast },e_{2}^{\ast })=1.
\end{eqnarray*}%
With these result the holomorphic curvature of the Iwasawa manifold $%
K(X,JX)=0.$

\end{document}